\documentclass[12pt]{amsart}
\usepackage{amsmath,amsthm,amsfonts,latexsym,amssymb,amscd,color,enumerate}
\usepackage{tikz}
\usetikzlibrary{patterns}
\usepackage{chemarr}

\newtheorem{theorem}{Theorem}[section]
\newtheorem{lemma}[theorem]{Lemma}
\newtheorem{corollary}[theorem]{Corollary}

\newtheorem{clm}[theorem]{Claim}
\newtheorem{case}{Case}
\newtheorem*{clm*}{Claim}

\newcommand{\cproof}{\noindent{\it Proof of Claim.}\ } 

\newcommand{\cqed}{\hfill\rule{1.3mm}{3mm}}

\theoremstyle{definition}

\theoremstyle{remark}

\numberwithin{equation}{section}

\begin{document}

\title{An uncountable
  J\'{o}nsson algebra in a minimal variety}

\author{Jordan DuBeau}
\address[Jordan DuBeau]{Department of Mathematics\\
University of Colorado\\
Boulder, CO 80309-0395\\
USA}
\email{jordan.dubeau@colorado.edu}

\author{Keith A. Kearnes}
\address[Keith Kearnes]{Department of Mathematics\\
University of Colorado\\
Boulder, CO 80309-0395\\
USA}
\email{keith.kearnes@colorado.edu}

\thanks{This material is based upon work supported by
  the National Science Foundation grant no.\ DMS 1500254.} 



\begin{abstract}
  We construct a J\'{o}nsson algebra of cardinality
  $\omega_1$ in the variety of J\'{o}nsson-Tarski algebras.
\end{abstract}

\maketitle

\section{Introduction}\label{intro_sec}
A \emph{J\'{o}nsson algebra} is  
an infinite algebra $J$ in a countable algebraic language,
which has no proper subalgebra of the same cardinality as $J$.
For example, the natural numbers, $\langle \mathbb N; S, 0\rangle$,
in the language of the successor function and $0$
is a J\'{o}nsson algebra.
Indeed, any infinite algebra that has no proper subalgebras at all
is a J\'{o}nsson algebra. Since J\'{o}nsson algebras have
countable languages, those
J\'{o}nsson algebras that have no proper subalgebras
must be countable.

Uncountable J\'{o}nsson algebras
are more difficult to construct
  and are more interesting.
  Keisler and Rowbottom announced in \cite{keisler-rowbottom}
  that if $V=L$, then there is a J\'{o}nsson algebra
  of every infinite cardinality. Erd\H{o}s and Hajnal
  showed in \cite{erdos-hajnal1} that if {\rm GCH} holds,
  then there is a J\'{o}nsson algebra of cardinality $\kappa^+$
  for every infinite cardinal $\kappa$. 
  Without {\rm GCH}, they proved that there is a
  J\'{o}nsson algebra of cardinality $\omega_n$ for every finite $n$.

  J\'{o}nsson algebras in specific varieties have been investigated.
  Some such results are surveyed in \cite{coleman}, but
  we list a few of the results here.
  Scott classified the J\'{o}nsson algebras in the variety
  of commutative groups
  in \cite{scott}: they are the Pr\"ufer $p^{\infty}$-groups,
  $\mathbb Z_{p^{\infty}}$,
  hence all have cardinality $\omega_0$.
  Shelah constructed J\'{o}nsson groups of cardinality $\omega_1$
  and of any cardinality $\lambda$ for which $\lambda^+=2^{\lambda}$
  in \cite{shelah1}. Ol'shanskii constructed
  Tarski Monsters in \cite{olshanskii}, which are
  special kinds of countable, noncommutative,
  J\'{o}nsson groups. McKenzie showed
  in \cite{mckenzie} that, under {\rm GCH},
  any J\'{o}nsson algebra in the variety of semigroups
  must be the underlying semigroup of a group.
  Hanf announced in \cite{hanf} that if there is
  a J\'{o}nsson algebra of cardinality $\kappa$, then a
  J\'{o}nsson algebra of cardinality $\kappa$ can be found
  in the variety of commutative loops.  
  Whaley proved in \cite{whaley} that there are no
  J\'{o}nsson algebras of regular cardinality in the variety
  of lattices.

  The question we consider is: Is there a J\'{o}nsson algebra
  in a minimal variety?
  This question is motivated by the observation
  that, if $J$ is one of the known J\'{o}nsson
  algebras, then some cyclic (= $1$-generated)
  subalgebras of $J$ seem
  to satisfy more identities than $J$ itself does.

  In Section~\ref{res_fin} of this paper
  we prove that there is no
  residually finite J\'{o}nsson algebra in a minimal variety
  (Corollary~\ref{RS2}).
  In Section~\ref{minimal} we prove that the variety of
  J\'{o}nsson-Tarski algebras is minimal
  (Corollary~\ref{minl}), and in Section~\ref{JJT}
  we describe an uncountable J\'{o}nsson algebra in the variety
  of J\'{o}nsson-Tarski algebras (Theorem~\ref{uncountable}).

\section{Residually finite J\'{o}nsson algebras}\label{res_fin}  

We started our investigation into whether there is a J\'{o}nsson algebra
in a minimal variety by examining
minimal \emph{locally finite} varieties, where much is known.
For example, the minimal, locally finite varieties
containing a nontrivial solvable member
have been classified,
\cite{kkv1,szendrei1,szendrei2}, and none 
contain a J\'{o}nsson algebra.
The minimal, locally finite, idempotent varieties
have been classified
\cite{szendrei3,szendrei4}, and none 
contain a J\'{o}nsson algebra.
In this section we will prove that there is no
residually finite J\'{o}nsson algebra in a minimal
variety, but first let us explain the connection
to J\'{o}nsson algebras in minimal locally finite varieties.

It is known that any minimal locally
finite variety $\mathcal V$ possesses a ``localization functor''
$e: {\mathcal V}\to e({\mathcal V})$, where the target
variety $e({\mathcal V})$ is a minimal, locally finite,
\emph{term minimal} variety (see \cite{kearnes-szendrei}).
Moreover, the target varieties
have been classified
into 20 types of minimal, locally finite,
term minimal varieties, \cite{szendrei-tm}.
This suggests a path to proving that there is no
J\'{o}nsson algebra in a minimal, locally finite variety.
First one should examine the 20 types
of minimal, locally finite varieties
of the form $e({\mathcal V})$, and prove that none of them
contains a J\'{o}nsson algebra.
Then one should show that the localization functor
reflects the absence of J\'{o}nsson algebras.

We were able to complete the first step for 19 of the 20 types,
but we were unable to complete this step for varieties
$e({\mathcal V})$ of the type
generated by strictly simple, $G^0$-algebras of type ${\bf 5}$,
whose complex nature is examined in \cite{kearnes-szendrei2}.
The method we used was to recognize that 19 of the
20 types of minimal, locally finite, term minimal varieties
consist of residually finite algebras. Then, 
we observed that there is no residually finite J\'{o}nsson
algebra in a minimal variety. We record the proof of this last statement
in this section.

The 20th type of minimal, locally finite, term minimal variety
need not consist of residually finite algebras, and we were
not able to determine whether there are varieties of this type which contain
J\'{o}nsson algebras.
We were also unable to prove that the localization functor
reflects the nonexistence of uncountable J\'{o}nsson algebras.
So, while it seems this circle of ideas
might still represent a viable path toward
proving the nonexistence of uncountable J\'{o}nsson
algebras in minimal, locally finite varieties,
in this section we shall only explain why
there are no residually finite J\'{o}nsson algebras in minimal varieties.

\begin{theorem}\label{theta}
  Let $\theta$ be a congruence on a J\'{o}nsson algebra $J$.
  \begin{enumerate}
  \item[$(1)$]
    If $|J/\theta|=|J|$, then $J/\theta$ is also a J\'{o}nsson algebra.
\item[$(2)$]
  If $|J/\theta|<\textrm{\rm cf}(|J|)$, then $J/\theta$ is
  cyclic.
\item[$(2)'$]
  If $\theta$ is a uniform congruence
  (all $\theta$-classes have the same size)
  and $|J/\theta|<|J|$, then $J/\theta$ is a
  cyclic algebra that is generated by any one of its elements.
 \end{enumerate}
\end{theorem}

\begin{proof}
  Let $\kappa=|J|$ and let $\nu: J\to J/\theta$ be the natural map.

  To prove the contrapositive of Item (1),
  assume $S\leq J/\theta$ is a proper subalgebra of 
  $J/\theta$ of size $\kappa$. Then $\nu^{-1}(S)$ is a proper subalgebra
  of size $\kappa$ of $J$, contradicting
  the assumption that $J$ is J\'{o}nsson.

  To prove Item~(2), note that the $\theta$-classes
  partition $J$ into $(<\textrm{cf}(\kappa))$-many subsets.
  Necessarily one of these classes, say $j/\theta$, has size $\kappa$.
  Since $J$ is J\'{o}nsson, the subalgebra $\langle j/\theta\rangle$
  generated by this class is $J$ itself. The Second Isomorphism
  Theorem implies that the quotient
  $J/\theta$ is generated by the single $\theta$-class $j/\theta$.

  Item~$(2)'$ is like Item~(2) with a minor difference.
  Let $\lambda=|J/\theta|<|J|=\kappa$ and let $\mu$
  be the class size for $\theta$.
  Necessarily, $\mu\cdot \lambda=\kappa>\lambda$,
  so $\mu=\kappa$ and every $\theta$-class $j/\theta$ has size $\kappa$.
  As in the proof of (2), this implies that every element of
  $J/\theta$ generates $J/\theta$. 
\end{proof}

\begin{corollary}
  Let $J$ be a J\'{o}nsson algebra which has a
  finite bound $n$ on the size of its cyclic subalgebras.
  $J$ is not residually finite.
\end{corollary}

\begin{proof}
  Suppose to the contrary that $J$ is residually finite
  and has a finite bound $n$ on the size of its cyclic subalgebras.
  Choose $n+1$ elements $a_i\in J$, $i=0,\ldots,n$.
  For each $i<j$ select a congruence $\theta_{ij}$
  on $J$ that has finite index and satisfies
  $(a_i,a_j)\notin \theta_{ij}$.
  If $\theta=\cap_{i<j} \theta_{ij}$,
  then $J/\theta$ is finite. By Theorem~\ref{theta}~(2),
  there is an element $a\in J$ such that
  $a/\theta$ generates $J/\theta$. 
  Hence
  \[
  |J/\theta|=|\langle a/\theta\rangle| = |\langle a\rangle /\theta|\leq
  |\langle a\rangle|\leq n.
  \]
  But $a_i/\theta\neq a_j/\theta$
  for any $i<j$, so $|J/\theta|\geq n+1$, which is a contradiction.
\end{proof}

The hypothesis of the previous corollary asserting
that $J$ has a finite bound on the size of its cyclic subalgebras
will hold if $J$ belongs to a variety whose free
algebra on one generator, $F_{\mathcal V}(1)$, is finite.
Call such a variety \emph{$1$-finite}.

\begin{corollary}\label{RS}
  A $1$-finite variety contains no residually finite
  J\'{o}nsson algebra. Hence a locally finite variety
  contains no residually finite
  J\'{o}nsson algebra. $\Box$
\end{corollary}

\begin{corollary}\label{RS2}
  There is no residually finite J\'{o}nsson algebra
  in a minimal variety.
\end{corollary}

\begin{proof}
  Suppose that $J$ is a residually finite J\'{o}nsson algebra
  in a minimal variety $\mathcal V$.
  Then $J$ has a nontrivial finite quotient,
  which necessarily generates $\mathcal V$,
  hence $\mathcal V$ is a finitely generated variety.
  Finitely generated varieties are locally finite,
  so this corollary follows from Corollary~\ref{RS}.
\end{proof}

In fact, we do not know if there are any uncountable,
residually finite J\'{o}nsson algebras at all.\footnote{%
  $\langle \mathbb N; S, 0\rangle$ and the unital ring
$\langle \mathbb Z; \cdot, +, -, 0, 1\rangle$ are examples
of countable, residually finite J\'{o}nsson algebras.}
The results above do place a cardinality
bound on the size of residually finite J\'{o}nsson algebras,
which we record.

\begin{theorem}
  If $J$ is a residually finite J\'{o}nsson algebra, then
  $|J|\leq 2^{2^{\omega}}$. If $J$ is defined in a finite language,
  then $|J|\leq 2^{\omega}$.
\end{theorem}

\begin{proof}
  Let $\mathcal V = {\mathcal V}(J)$ be the variety generated by $J$,
  and let $A\in {\mathcal V}$ be any finite member.
  All cyclic algebras
  in the subvariety ${\mathcal V}(A)\leq {\mathcal V}$ have size
  bounded by $|F_{{\mathcal V}(A)}(1)|\leq |A|^{|A|}$, which is finite.
  Since the finite quotients of $J$ are cyclic,
  $J$ has a maximal finite quotient that lies in this subvariety.
  That is, $J$ has a least congruence $\theta_A$ such that
  $J/\theta_A\in {\mathcal V}(A)$ (and $|J/\theta_A|\leq |A|^{|A|}$ will hold).
  Now $J$ is embeddable in
  \[\prod_{A\in {\mathcal V}, |A|<\omega} J/\theta_A,\]
  since $J$ is residually finite. The factors in this product
  are finite, while the number of factors in this displayed product equals
  the number of finite algebras in ${\mathcal V}$.
  A variety in a countable language has at most
  $2^{\omega}$ finite members, leading to the cardinality bound
  \[|J|\leq \left|\prod_{A\in {\mathcal V}, |A|<\omega} J/\theta_A\right|
  \leq
  \omega^{2^{\omega}}=
  2^{2^{\omega}}\]
  when the language is countable.
  A variety in a finite language has at most
  countably many finite members, so the same calculation produces
  the bound
  $|J|\leq 2^{\omega}$ when the language is finite.
\end{proof}

We shall leave open the following questions.

\bigskip

\noindent
{\bf Question 1.} Is there 
an uncountable, residually finite J\'{o}nsson algebra?

\bigskip

\noindent
{\bf Question 2.} Is there an uncountable J\'{o}nsson
algebra in a minimal, locally finite variety?

\bigskip

Henceforth we concentrate on
constructing a non-locally finite, minimal variety
that contains an uncountable J\'{o}nsson algebra.
Specifically, we construct a J\'{o}nsson algebra
of cardinality $\omega_1$ in the variety
of J\'{o}nsson--Tarski algebras.

\section{The variety of J\'{o}nsson-Tarski algebras is minimal}\label{minimal}
A \textit{J\'{o}nsson-Tarski algebra} is an algebra
$\langle A; \cdot, \ell, r \rangle$ in the language $\mathcal L$ of
one binary operation $\cdot$ and two unary operations
$\ell$ and $r$, which satisfies identities expressing 
that
\[
A\times A\to A: (x,y)\mapsto x\cdot y,\quad\textrm{and} \quad
A\to A\times A: z\mapsto (\ell(z),r(z))
\]
are inverse bijections. These identities are 
\begin{enumerate}
\item[$\varepsilon({\ell})$:]\quad $\ell(x \cdot y) = x$,
\item[$\varepsilon({r})$:]\quad  $r(x \cdot y) = y$, and
\item[$\varepsilon({\cdot})$:]\quad  $\ell(z) \cdot r(z) = z$.
\end{enumerate}
Let $\Sigma=\{\varepsilon({\ell}), \varepsilon({r}), \varepsilon({\cdot})\}$
denote the set of these identities.

J\'{o}nsson and Tarski introduced the variety
axiomatized by $\Sigma$ in \cite{jonsson-tarski}.
J\'{o}nsson and Tarski were interested in this variety
because it has the property 
that the algebra $F_{\mathcal V}(n)$ freely generated by
the $n$-element set $X=\{x_1,\ldots,x_{n-1},x_n\}$, $n > 1$,
is also freely generated by the $(n-1)$-element set
$\{x_1,\ldots,(x_{n-1}\cdot x_n)\}$, where the last two
elements of $X$ are replaced by their product.
Consequently $F_{\mathcal V}(n)\cong F_{\mathcal V}(n-1)$
in this variety, and in fact any two free algebras with
a finite, positive number of generators are isomorphic.
We mention this for information only.
We shall not use anything from \cite{jonsson-tarski}
other than the definition of the variety of J\'{o}nsson-Tarski algebras.

We call the language of the binary symbol $\cdot$
the ``multiplication sublanguage'', $\mathcal L_m$,
and the language of the two unary symbols
only, $\ell$ and $r$, the ``unary sublanguage'', $\mathcal L_u$.
Call an $\mathcal L$-term an ``$m,u$-term''
if it has the form
$w(u_1(x_{i_1}),\ldots,u_k(x_{i_k}))$
where $w(x_1,\ldots,x_k)$ is an $\mathcal L_m$-term
and the $u_j(x_{i_j})$
are $\mathcal L_u$-terms.
Let $\mathcal {MU}$ denote the set of $m,u$-terms
in the language $\mathcal L$.

\begin{lemma}\label{NF1}
Every $\mathcal L$-term is $\Sigma$-equivalent to
a term in $\mathcal {MU}$.
\end{lemma}

\begin{proof}
  The set of $\mathcal L$-terms is the smallest
  set containing the variables and closed under
  multiplication and $\ell$ and $r$. 
  $\mathcal {MU}$ 
  contains the variables and is closed under multiplication,
  so we only need to show that, for any $t\in \mathcal {MU}$,
  $\ell(t)$ and $r(t)$ are $\Sigma$-equivalent to
  terms in $\mathcal {MU}$.
  
  If $t$ contains
  at least one occurrence of the multiplication symbol,
  then $t = p\cdot q$ where $p, q\in \mathcal {MU}$.
  Then $\ell(t)$ and $r(t)$
  are $\Sigma$-equivalent to 
  $p$ and $q$ respectively, which belong to $\mathcal {MU}$.
  If $t$ does not contain
  an occurrence of the multiplication symbol,
  then $t$ is a term in the sublanguage $\mathcal L_u$.
  In this case $\ell(t)$ and $r(t)$ are also
  terms in the sublanguage $\mathcal L_u$, so they are
  in $\mathcal {MU}$.
\end{proof}  

\begin{lemma} \label{NF2}
  If $s$ and $t$ are $\mathcal L$-terms, then either
  \begin{enumerate}
\item    
  $\Sigma$ entails the identity $s = t$, or
\item    
  $\Sigma\cup\{s=t\}$ entails $x=y$.
  \end{enumerate}
\end{lemma}

\begin{proof}
  We argue by induction on the total number of
  multiplication symbols occurring in $s=t$.
  
  Replace $s$ and $t$ with $\Sigma$-equivalent
  $m,u$-terms, each involving a minimal number of
  multiplication symbols. This does not affect the assumptions
  nor the conclusions of the theorem. Now we consider cases.

  \begin{case}\label{case1}
At least one of $s$ or $t$ has at least one multiplication symbol.
  \end{case}

\begin{caseproof}  
  In this case, the identities
  $\ell(s)=\ell(t)$ and $r(s)=r(t)$, which
  are both derivable from $s=t$, each
  have fewer total multiplication symbols than $s=t$ after
  reducing modulo $\Sigma$.
  By induction we have either
  \begin{enumerate}
\item[(I)]
  $\Sigma$ entails $\ell(s) = \ell(t)$ and $r(s)=r(t)$, or
\item[(II)] 
  $\Sigma\cup\{\ell(s)=\ell(t)\}$ entails $x=y$ or
  $\Sigma\cup\{r(s)=r(t)\}$ entails $x=y$.
  \end{enumerate}
  Under Item (I), $\Sigma$ entails
  $s = \ell(s)\cdot r(s)=\ell(t)\cdot r(t) = t$, and we get Item (1)
  of the theorem statement.
  Under Item (II), $\Sigma\cup\{s=t\}$ entails
  both $\ell(s)=\ell(t)$ and $r(s)=r(t)$, hence entails $x=y$.
  This gives us Item (2)
  of the theorem statement.
  \end{caseproof}
    
  \bigskip

  For the rest of the proof we may restrict attention
  to the case where $s$ and $t$ are $\mathcal L_u$-terms.
  We write a term like $\ell(\ell(r(\ell(r(r(z))))))$
  as $\ell\ell r\ell rr(z)$ and refer to $z$
  as the variable of the term. We refer  
  to the sequence of $\ell$'s
  and $r$'s (which is $\ell\ell r\ell rr$
  in this example) as the prefix. We refer 
  to the number of symbols in the prefix
    (which is six in this example) as the length
  of the term. We refer to the underlined
  symbol $\ell\ell r\ell r\underline{r}(z)$
  as the rightmost operation symbol of the term.
  Without loss of generality, we assume that the 
  length of $s$ is at least that of $t$.

  \begin{case}\label{case2}
$s$ and $t$ have different variables and the same prefix.
  \end{case}
  
  \begin{caseproof}
    In this case we show by induction on the combined
    lengths of $s$ and $t$ that $\Sigma\cup \{s=t\}$
  entails $x=y$. If the common prefix of $s$ and $t$ is empty, 
  then $s$ and $t$ are already distinct variables, so
  $\Sigma \cup \{s=t\}$ entails $x=y$. Otherwise we may assume that
  $s$ is $PQ(x)$ and $t$ is $PQ(y)$ where $PQ$ is the 
  prefix, $P$ is a possibly empty string, and $Q$ is either $\ell$ or $r$.
  Now, by substituting the terms 
  $x_{\ell}\cdot x_r$ and   $y_{\ell}\cdot y_r$ for $x$ and $y$,
  where $x_{\ell}, x_r, y_{\ell}, y_r$ are
  distinct new variables,
  we find that 
  $\Sigma\cup \{s=t\}$ entails
  $PQ(x_{\ell}\cdot x_r) = PQ(y_{\ell}\cdot y_r)$,
  which reduces
  modulo $\Sigma$ to $P(x_Q)=P(y_Q)$. By induction,
  $\Sigma\cup \{P(x_Q)=P(y_Q)\}$ entails $x=y$, so we
  conclude that $\Sigma\cup \{s=t\}$
  entails $x=y$. 
  \end{caseproof}
    
  \bigskip
  
  \begin{case}\label{case3}
$s$ and $t$ have different variables and different prefixes.
  \end{case}

\begin{caseproof}    
  In this case we also show that $\Sigma\cup \{s=t\}$
  entails $x=y$.
 We may assume that $s=t$ is expressible as $P(x)=P'(y)$, where
 $P$ is nonempty and is the prefix of $s$, while $P'$ may be an empty
 string and is the prefix of $t$.
Substituting $z$ for $y$ we derive
$P(x)=P'(z)$ from $P(x)=P'(y)$. From $P'(y)=P(x)=P'(z)$
we derive the identity
$P'(y)=P'(z)$.
Citing Case~\ref{case2} we conclude that
$\Sigma\cup \{s=t\}$ entails $x=y$.
\end{caseproof}
  
\bigskip
  
  \begin{case}\label{case4}
    $s$ and $t$ have the same variables, but different
    rightmost operation symbols.
  \end{case}

  \begin{caseproof}
    In this case we express $s=t$ as $PQ(x)=P'Q'(x)$
  where $Q$ is a single operation symbol, $\ell$ or $r$, 
  $Q\neq Q'$, 
  the length of $PQ(x)$ is at least that of $P'Q'(x)$,
  and we allow $P'Q'$ to be empty.

  We first treat the case where $P'Q'$ is not empty,
  so $\{Q, Q'\}=\{\ell, r\}$.
  Substituting $x_{\ell}\cdot x_r$ for $x$, where $x_{\ell}$
  and $x_r$ are new variables, we derive from 
  $PQ(x)=P'Q'(x)$ the identity $P(x_{Q})=P'(x_{Q'})$.
  This is an identity of the type handled in either
  Case 2 or Case 3.
  Hence 
  $\Sigma\cup\{s=t\}$ 
  entails $x=y$.

  Next we treat the case where $P'Q'$ is empty, so
  $s=t$ is $PQ(x) = x$. We assume that $Q=\ell$,
  and omit the argument for the similar case where $Q=r$.
  Thus, it is our aim to show that $\Sigma\cup\{P\ell(x)=x\}$
  entails $x=y$. Substitute $x_{\ell}\cdot x_r$ for $x$
  in $P\ell(x)=x$ and apply $r$ to both sides to obtain
  the second equality in 
  \[
rP (x_{\ell}) = r P \ell (x_{\ell}\cdot x_r) = r(x_{\ell}\cdot x_r) = x_r.
\]
Thus, from
$\Sigma\cup\{P\ell(x)=x\}$ we have derived $rP(x_{\ell})=x_r$,
which is an identity of the type considered in Case~\ref{case3}.
From the Case~\ref{case3} argument we derive $x=y$.
\end{caseproof}
  
\bigskip

\begin{case}\label{case5}
    $s$ and $t$ have the same variables, and the same
    rightmost operation symbols.
  \end{case}

\begin{caseproof}  
If the prefixes of $s$ and $t$ are nonempty, then $s=t$ 
is expressible as $PQ(x) = P'Q(x)$
where $Q$ is a single operation symbol, $\ell$ or $r$. 
By substituting $x_{\ell}\cdot x_r$ for $x$ and 
reducing modulo $\Sigma$, we obtain 
$P(x_Q) = P'(x_Q)$. Now either we are in Case 4, 
from which it follows that $\Sigma \cup \{s=t\}$ entails 
$x=y$, or else we are back in Case 5 but with terms 
of strictly shorter length. If we are back in Case 5, we repeat the argument, and
we either eventually reach Case 4, or we reach an expression
of the form $x=x$ and conclude that $s$ and $t$ 
had exactly the same prefixes and the same variables, hence
$\Sigma$ entails $s=t$. Thus the assertion
of the lemma holds.
\end{caseproof}  
\end{proof}  

\begin{corollary} \label{minl}
The variety of J\'{o}nsson-Tarski algebras is minimal.
\end{corollary}

\begin{proof}
  If the variety $\mathcal J$ axiomatized by $\Sigma$
  were not minimal, then
  there would exist an identity $s=t$
  that does not hold throughout $\mathcal J$,
  but which holds in some minimal subvariety of $\mathcal J$.
    For this identity
  we would have $\Sigma\not\models s=t$ and 
  $\Sigma\cup\{s=t\}\not\models x=y$, contrary to
  Lemma~\ref{NF2}.
\end{proof}

\section{J\'{o}nsson J\'{o}nsson-Tarski algebras}\label{JJT}

Our goal in this section is to construct a
J\'{o}nsson algebra on the set $\omega_1$
in the variety of J\'{o}nsson-Tarski algebras.
We start with the easier project of constructing a
J\'{o}nsson J\'{o}nsson-Tarski algebra structure on $\omega$.

To construct a J\'{o}nsson-Tarski algebra
on a set $J$, it suffices to describe the
table for the multiplication operation,
since it is possible to read the tables for $\ell$
and $r$ off of the multiplication table.
Moreover, the only condition that a multiplication
table must satisfy for it to be a table for a
J\'{o}nsson-Tarski algebra is that it be
the table of a bijection
$J\times J\to J$, which means that every element of $J$
occurs in one and only one cell of the
multiplication table.

\begin{theorem}\label{ctblejonssonthm}
  There exists a J\'{o}nsson J\'{o}nsson-Tarski algebra,
  $J_{\omega}$, with universe $\omega$.
\end{theorem}

\begin{proof}
  We define our J\'{o}nsson J\'{o}nsson-Tarski algebra
with the following multiplication table:
  
\bigskip

\begin{center}
\begin{tabular}{|c || c | c | c | c | c | c}
\hline
$\cdot$ & 0 & 1 & 2 & 3 & 4 & $\cdots$ \\
\hline
\hline
0 &  1 & 2 & 5 & 9 & 14 & $\cdots$ \\
\hline
1 & 0 & 4 & 8 & 13 & 19 & $\cdots$ \\
\hline
2 & 3 & 7 & 12 & 18 & 25 & $\cdots$ \\
\hline
3 & 6 & 11 & 17 & 24 & 32 & $\cdots$ \\
\hline
4 & 10 & 16 & 23 & 31 & 40 & $\cdots$\\
\hline
$\vdots$ & $\vdots$ & $\vdots$ & $\vdots$ & $\vdots$ & $\vdots$ & $\ddots$
\end{tabular}
\end{center}

\bigskip

This table was created by first
placing the numbers $0, 1$, and $2$ in the upper left corner
in a certain pattern. The rest of the table
was filled by placing the remaining natural numbers
in the remaining empty cells in increasing order moving
diagonally up and to the right.

In this table, the product $m\cdot n$
is placed in the cell located in the $m$-th row and $n$-th column.
For example, $3\cdot 4=32$, by definition, so $32$ is placed in
the cell in the $3$rd row and $4$th column.
To read off the $\ell$ and $r$ operations
from this multiplication
table, recall that any $n\in\omega$ occurs in one and only 
one cell of the table. The value of $\ell(n)$ is the row header
for that cell containing $n$,
while $r(n)$ is the column header for that cell.
For example, $\ell(13)=1$ and $r(23)=2$.

An important feature of this algebra, $J_{\omega}$, is that
the functions $\ell$ and $r$ are regressive.
That is, $\ell(n)<n$ and $r(n)<n$ when $n\neq 0$.
This can be checked by hand for $n=1$ and $2$,
and then proved for larger $n$ using the following formula
for the multiplication
\[
p\cdot q = \binom{p+q+1}{2} + q,\quad {\rm if}\; p+q>1,
\]
which is not hard to establish.

An immediate consequence of regressiveness is:

\begin{clm}
  If $n \in \omega$, then $0$ is in the subalgebra
  of $J_{\omega}$ generated by $\{n\}$.
\end{clm}
\begin{cproof}
  Since $\ell$ is regressive,
  the sequence $n, \ell(n), \ell(\ell(n)), \dots$
  must eventually include $0$.
\end{cproof}
\bigskip

\begin{clm}
$J_{\omega}$ is generated by $\{0\}$.
\end{clm}

\begin{cproof}
  Suppose otherwise, and let $n\in\omega$ be the least
  natural number not in the subalgebra $S=\langle \{0\}\rangle$.
  Clearly $n>0$, so $\ell(n), r(n)<n$, which implies
  that $\ell(n), r(n)\in S$.
  But now $n = \ell(n)\cdot r(n)\in S$, contradicting the choice of $n$.
\end{cproof}

\bigskip

The claims show that any $n\in\omega$ generates $J_{\omega}$,
so $J_{\omega}$ has no proper (nonempty) subalgebras.
This establishes that $J_{\omega}$ is J\'{o}nsson.
\end{proof}

It is impossible to construct a
J\'{o}nsson J\'{o}nsson-Tarski algebra $J_{\omega_1}$
on $\omega_1$ that has
the same properties as $J_{\omega}$,
namely the properties that
$\ell$ and $r$ are regressive functions,
and that each element of $J_{\omega}$ generates the whole algebra.
In the first place, single elements can only generate
countable subalgebras, so the uncountable algebra $J_{\omega_1}$
cannot be $1$-generated (or even countably generated).
In the second place, as pointed out to us by Don Monk,
Fodor's Lemma prevents the existence of a
J\'{o}nsson J\'{o}nsson-Tarski algebra on $\omega_1$ which has both
$\ell$ and $r$ regressive.
To see this, note that if $\ell$ is regressive on $\omega_1$,
then it is constant on a set $S$ stationary in $\omega_1$.
If $r$ is also regressive on $\omega_1$, so regressive on $S$,
then $r$ is constant on a stationary subset $T\subseteq S$.
If $\ell$ and $r$ are both constant on $T$, then the map
$\omega_1\to \omega_1\times\omega_1: \alpha\mapsto (\ell(\alpha),r(\alpha))$
is constant on $T$. But in a J\'{o}nsson-Tarski
algebra this map is a bijection.

Nevertheless, we do have:

\begin{theorem} \label{uncountable}
  There exists a J\'{o}nsson J\'{o}nsson-Tarski algebra,
  $J_{\omega_1}$, with universe $\omega_1$.
\end{theorem}

\begin{proof}
We will construct the multiplication table for
$J_{\omega_1}$ by transfinite recursion.
The main part of the argument will involve
explaining, for each countable limit ordinal
$\lambda$, how to extend a given
J\'{o}nsson-Tarski multiplication on $\lambda$
to a J\'{o}nsson-Tarski
multiplication on $\lambda+\omega$,
thereby enlarging a
J\'{o}nsson-Tarski algebra $J_{\lambda}$
to a J\'{o}nsson-Tarski superalgebra $J_{\lambda+\omega}$.
Our construction will be guided by the desire
to maintain the following Property $\Omega$:

  \begin{quote}
    For each pair of countable limit ordinals $\alpha<\beta$
    the subalgebra of $J_{\beta}$ generated
    by any element of the form $\alpha+n$, $n\in\omega$, is $\alpha+\omega$.
    \end{quote}

  
  Theorem~\ref{ctblejonssonthm} provides
  a J\'{o}nsson-Tarski multiplication  defined on $\omega$ 
  which satisfies Property $\Omega$.
  
  Now assume that $\lambda$ is an infinite countable limit ordinal, and
  that we have a J\'{o}nsson-Tarski multiplication
  with Property $\Omega$ defined on $\lambda$.
  We explain how to extend the $\lambda\times\lambda$
  J\'{o}nsson-Tarski multiplication
  table that has Property $\Omega$ to a 
  $(\lambda+\omega)\times (\lambda+\omega)$
  multiplication
  table that has Property $\Omega$.
  
  Call $\lambda+n$ ``even'' if $n$ is even,
  and ``odd'' if $n$ is odd. Say that $\lambda+n$ is
  ``divisible by $4$'' or is 
  ``$0\pmod{4}$'' if $n\equiv 0\pmod{4}$. ETC.
  
  We first explain where to place the value $\lambda+n$ in the
  multiplication table of $J_{\lambda+\omega}$
  when $\lambda+n$ is divisible by $4$.
  Place this ordinal in the cell whose column header
  is $r(\lambda+n)=\lambda+n+2$, and whose row header
  is specified as follows. Choose $\ell$ to be any bijection from the
  countably infinite set of ordinals in
  $\{\lambda+k \mid k\equiv 0\pmod{4}\}$
  onto the set $\lambda$,
  and then the row header for the cell containing $\lambda+n$
  will be $\ell(\lambda+n)$.

  For the ordinals $\lambda+n$ that are $2\pmod{4}$,
  we again define $r(\lambda + n) = \lambda + n + 2$.
  Then we let $\ell$ be the map defined by
  $\ell(\lambda + n) = \lambda + \frac{n-2}{4}$.
  The function $\ell$ maps the set of ordinals that are $2\pmod{4}$
  bijectively onto the set $\{\lambda + k \mid k \in \omega\}$.

  So far, we have defined $\ell$ so that it is a
  bijection from the set of even ordinals $\geq \lambda$
  onto the set $\lambda + \omega$, while $r(\lambda+n)=\lambda+n+2$ holds
  whenever $\lambda+n$ is even. The multiplication table
  as described so far is indicated in Figure~\ref{jtpicture1}.
  The placement in the figure of the ordinals that are divisible by $4$
  is only suggestive: they are shown in the correct columns,
  but all we know about which rows they occupy is that there
  is one ordinal of the form $\lambda+n$, $n\equiv 0\pmod{4}$,
  in each row whose row header is $< \lambda$.

\begin{figure}[ht]
\begin{tikzpicture}[scale=0.7]
\draw [thick, pattern = north east lines] (0,16) rectangle (3,13);
\draw [thick, pattern = north east lines ] (3.5,16) rectangle (6.5,13);
\draw [thick, pattern = north east lines] (0,12) rectangle (3,9);
\draw [thick, pattern = north east lines] (3.5,12) rectangle (6.5,9);
\draw [thick] (0,8) rectangle (3,2);
\draw [thick] (3.5, 8) rectangle (6.5,2);
\draw [thick] (7,8) rectangle (20,2);
\draw [thick] (7, 12) rectangle (20, 9);
\draw [thick](7, 16) rectangle (20, 13);

\foreach \y in {3,...,7}
{
	\draw [thin] (0, \y) -- (3, \y);
	\draw [thin] (3.5, \y) -- (6.5, \y);
	\draw [thin] (7, \y) -- (20, \y);
}
\foreach \x in {8,...,19}
{
	\draw [thin] (\x, 8) -- (\x, 2);
	\draw [thin] (\x, 12) -- (\x, 9);
	\draw [thin] (\x, 16) -- (\x, 13);
}
\foreach \y in {10,11}
	\draw [thin] (7, \y) -- (20, \y);
\foreach \y in {14,15}
	\draw [thin] (7, \y) -- (20, \y);
\foreach \x in {1,2}
	\draw [thin] (\x, 2) -- (\x, 8);
\foreach \x in {4,...,7}
{
	\draw [thin] (\x - 0.5, 2) -- (\x - 0.5, 8);
}

\node at (7.5, 16.5) {\tiny $\lambda$};
\foreach \x in {0,1,2}
	\node at (\x+.5, 16.5) {\small \x};
\node at (4, 16.5) {\tiny $\omega$};
\foreach \x in {5,...,6}
{
	\pgfmathtruncatemacro{\n}{\x-4};
	\node at (\x, 16.5) {\tiny $\omega$+\n};
}
\foreach \x in {8,...,19}
{
	\pgfmathtruncatemacro{\n}{\x-7};
	\node at (\x+.5, 16.5) {\tiny $\lambda$+\n};
}
\foreach \y in {13,14,15}
{
	\pgfmathtruncatemacro{\n}{15-\y};
	\node at (-0.5, \y + 0.5) {\small \n};
}
\node at (-0.6, 11.5) {\tiny $\omega$};
\foreach \y in {9,...,10}
{
	\pgfmathtruncatemacro{\n}{11-\y};
	\node at (-0.6, \y + 0.5) {\tiny $\omega$+\n};
}
\node at (-0.6, 7.5) {\tiny $\lambda$};
\foreach \y in {2,...,6}
{
	\pgfmathtruncatemacro{\n}{7-\y};
	\node at (-0.6, \y + 0.5) {\tiny $\lambda$+\n};
}

\node at (3.3, 16.5) {\small $\ldots$};
\node at (6.8, 16.5) {\small $\ldots$};
\node at (-0.5, 12.6) {$\vdots$};
\node at (-0.5, 8.6) {$\vdots$};

\node at (11.5, 7.5) {\tiny $\lambda$+2};
\node at (15.5, 6.5) {\tiny $\lambda$+6};
\node at (19.5, 5.5) {\tiny $\lambda$+10};
\node at (9.5, 15.5) {\tiny $\lambda$};
\node at (13.5, 11.5) {\tiny $\lambda$+4};
\node at (17.5, 14.5) {\tiny $\lambda$+8};
\end{tikzpicture}
\bigskip
\caption{The J\'{o}nsson-Tarski algebra $J_{\lambda + \omega}$ after placing even ordinals.}
\label{jtpicture1}
\end{figure}

Finally, we use the odd ordinals to
populate the rest of this
$(\lambda + \omega) \times (\lambda + \omega)$ table.
We begin by partitioning the set of odd ordinals
above $\lambda$ into countably many countable sets
labelled $L_{\lambda+n}$ for $n \in \omega$, as indicated
by the columns of Table~\ref{Ltable}.


\begin{table}[ht]
\centering
\begin{tabular}{c | c | c | c | c}
\hline
\hphantom{AA}$L_{\lambda}$\hphantom{AA} & $L_{\lambda+1}$ & $L_{\lambda+2}$ & $L_{\lambda+3}$ & $\cdots$ \\ [0.5ex]
\hline
\hline 
$\lambda+1$ & $\lambda+3$ &  &  & \\ 
$\lambda+5$ & $\lambda+7$ & $\lambda+9$ &  & \\
$\lambda+11$ & $\lambda+13$ & $\lambda+15$ & $\lambda+17$ & \\
$\lambda+19$ & $\lambda+21$ & $\lambda+23$ & $\lambda+25$ & $\cdots$\\
\end{tabular}
\bigskip

\caption{Partition of odd ordinals into the sets $L_{\lambda+n}$.}
\label{Ltable}
\end{table}

This way of partitioning of the odd ordinals above $\lambda$
has the property that $\lambda+n \in L_{\lambda+m}$ implies $m < n$.

Let each set $L_{\lambda+n}$
fill the L-shaped region of the table in Figure~\ref{jtpicture2}
corresponding to $\lambda+n$.
By that, we mean the region containing
all cells at addresses $(\lambda+n, \alpha)$
and $(\alpha, \lambda+n)$ for all $\alpha \leq \lambda+n$.
We can do this since the set of these cells is
countably infinite, and so is $L_{\lambda+n}$.
It could be, from our placement of the even ordinals,
that some of the cells in this region are already occupied.
However, it can be easily seen that at most two 
cells in each L-shaped region will be occupied by even ordinals,
since no two even ordinals
were assigned the same $\ell$-value and no two even
ordinals were assigned the same $r$-value.
Thus, each L-shaped region
will still have infinitely many cells
in which to place the elements of $L_{\lambda+n}$.
See Figure~\ref{jtpicture2}.

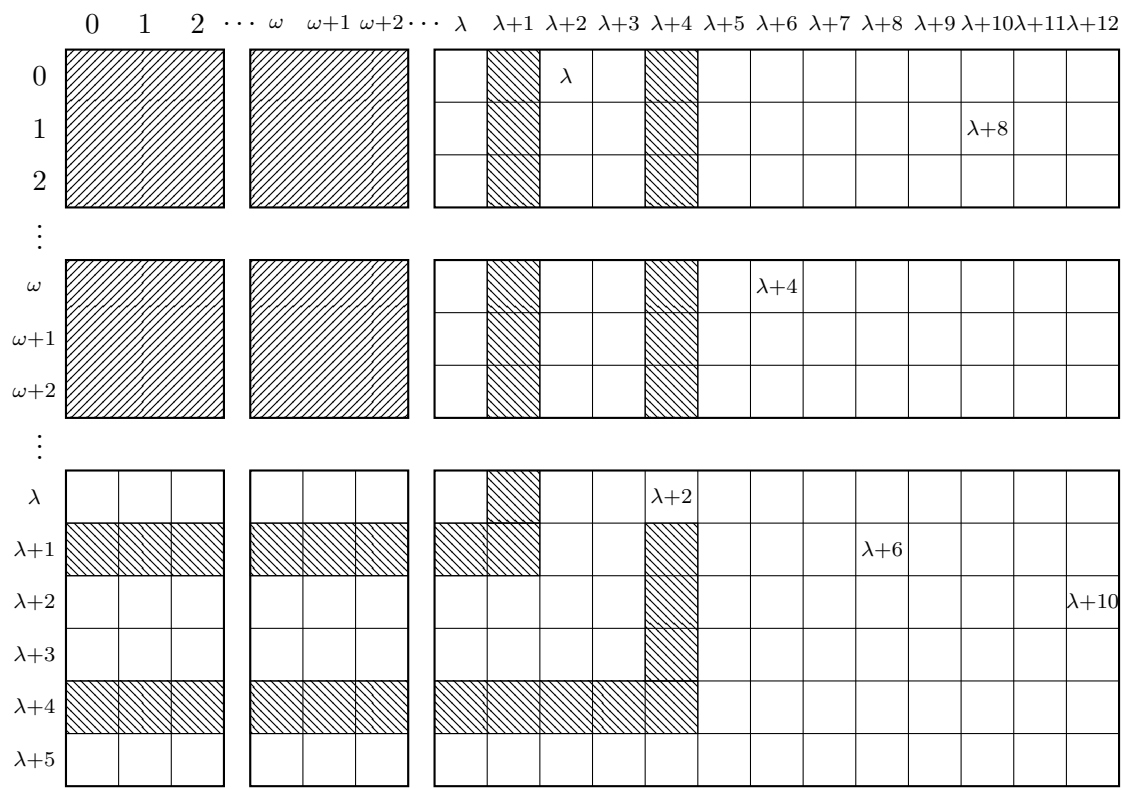
\begin{figure}[ht] 
\begin{tikzpicture}[scale=0.7]
\draw [thick, pattern = north east lines] (0,16) rectangle (3,13);
\draw [thick, pattern = north east lines ] (3.5,16) rectangle (6.5,13);
\draw [thick, pattern = north east lines] (0,12) rectangle (3,9);
\draw [thick, pattern = north east lines] (3.5,12) rectangle (6.5,9);
\draw [thick] (0,8) rectangle (3,2);
\draw [thick] (3.5, 8) rectangle (6.5,2);
\draw [thick] (7,8) rectangle (20,2);
\draw [thick] (7, 12) rectangle (20, 9);
\draw [thick](7, 16) rectangle (20, 13);

\foreach \y in {3,...,7}
{
	\draw [thin] (0, \y) -- (3, \y);
	\draw [thin] (3.5, \y) -- (6.5, \y);
	\draw [thin] (7, \y) -- (20, \y);
}
\foreach \x in {8,...,19}
{
	\draw [thin] (\x, 8) -- (\x, 2);
	\draw [thin] (\x, 12) -- (\x, 9);
	\draw [thin] (\x, 16) -- (\x, 13);
}
\foreach \y in {10,11}
	\draw [thin] (7, \y) -- (20, \y);
\foreach \y in {14,15}
	\draw [thin] (7, \y) -- (20, \y);
\foreach \x in {1,2}
	\draw [thin] (\x, 2) -- (\x, 8);
\foreach \x in {4,...,7}
{
	\draw [thin] (\x - 0.5, 2) -- (\x - 0.5, 8);
}

\node at (7.5, 16.5) {\tiny $\lambda$};
\foreach \x in {0,1,2}
	\node at (\x+.5, 16.5) {\small \x};
\node at (4, 16.5) {\tiny $\omega$};
\foreach \x in {5,...,6}
{
	\pgfmathtruncatemacro{\n}{\x-4};
	\node at (\x, 16.5) {\tiny $\omega$+\n};
}
\foreach \x in {8,...,19}
{
	\pgfmathtruncatemacro{\n}{\x-7};
	\node at (\x+.5, 16.5) {\tiny $\lambda$+\n};
}
\foreach \y in {13,14,15}
{
	\pgfmathtruncatemacro{\n}{15-\y};
	\node at (-0.5, \y + 0.5) {\small \n};
}
\node at (-0.6, 11.5) {\tiny $\omega$};
\foreach \y in {9,...,10}
{
	\pgfmathtruncatemacro{\n}{11-\y};
	\node at (-0.6, \y + 0.5) {\tiny $\omega$+\n};
}
\node at (-0.6, 7.5) {\tiny $\lambda$};
\foreach \y in {2,...,6}
{
	\pgfmathtruncatemacro{\n}{7-\y};
	\node at (-0.6, \y + 0.5) {\tiny $\lambda$+\n};
}

\node at (3.3, 16.5) {\small $\ldots$};
\node at (6.8, 16.5) {\small $\ldots$};
\node at (-0.5, 12.6) {$\vdots$};
\node at (-0.5, 8.6) {$\vdots$};

\node at (11.5, 7.5) {\tiny $\lambda$+2};
\node at (15.5, 6.5) {\tiny $\lambda$+6};
\node at (19.5, 5.5) {\tiny $\lambda$+10};
\node at (9.5, 15.5) {\tiny $\lambda$};
\node at (13.5, 11.5) {\tiny $\lambda$+4};
\node at (17.5, 14.5) {\tiny $\lambda$+8};

\draw [thin, pattern = north west lines] (0,7) rectangle (3,6);
\draw [thin, pattern = north west lines] (3.5,7) rectangle (6.5,6);
\draw [thin, pattern = north west lines] (7,7) rectangle (9,6);
\draw [thin, pattern = north west lines] (8,8) rectangle (9,7);
\draw [thin, pattern = north west lines] (8,12) rectangle (9,9);
\draw [thin, pattern = north west lines] (8,16) rectangle (9,13);

\draw [thin, pattern = north west lines] (0,4) rectangle (3,3);
\draw [thin, pattern = north west lines] (3.5,4) rectangle (6.5,3);
\draw [thin, pattern = north west lines] (7,4) rectangle (12,3);
\draw [thin, pattern = north west lines] (11,7) rectangle (12,4);
\draw [thin, pattern = north west lines] (11,12) rectangle (12,9);
\draw [thin, pattern = north west lines] (11,16) rectangle (12,13);
\end{tikzpicture}
\bigskip
\caption{Placement of the odd ordinals. The set $L_{\lambda+1}$ will fill the L-shaped region corresponding to $\lambda+1$, and similarly for the set $L_{\lambda+4}$.}
\label{jtpicture2}
\end{figure}

Once the elements
of the sets $L_{\lambda+n}$ have been placed in the table,
we will have filled the square
$(\lambda + \omega) \times (\lambda + \omega)$.
Moreover, we have arranged that, whenever
$\lambda + n \in L_{\lambda+m}$ we have $m < n$,
and we either have $\ell(\lambda+n) = \lambda+m$
or $r(\lambda+n) = \lambda+m$. We now argue
that this construction extends the
original $\lambda\times \lambda$ table in a  way that preserves
Property $\Omega$.

We assume that Property $\Omega$ holds for the
subalgebra $J_{\lambda}$ of $J_{\lambda+\omega}$, so
to show that Property $\Omega$ holds for $J_{\lambda+\omega}$ 
it suffices to show that any $\lambda + n$ generates the set $\lambda+\omega$.
As a first step, we argue that any $\lambda + n$
generates $\lambda$. Toward this end, we argue
that any $\lambda+n$, $n>0$, generates some $\lambda + m$, where $m < n$. 
If $n$ is odd, then
$\lambda + n$ belongs to $L_{\lambda+m}$ for some $m < n$,
hence either $\ell(\lambda + n) = \lambda+m$ or $r(\lambda + n) = \lambda+m$. 
If $n$ is $2\pmod{4}$, then $\ell(\lambda + n) = \lambda + \frac{n-2}{4}$,
and $\frac{n-2}{4} < n$.
Finally, if $n$ is $0\pmod{4}$,
then $\ell(r(\lambda + n)) = \ell(\lambda + n + 2) = \lambda + \frac{(n+2)-2}{4} = \lambda + \frac{n}{4}$,
and $\frac{n}{4} < n$. This completes the proof
that any $\lambda + n$ generates some smaller $\lambda + m$,
hence generates $\lambda$.

Now, $\lambda$ generates the set of all even ordinals
between $\lambda$ and $\lambda + \omega$, since
$r$ adds $2$ to any even ordinal.
Next, $\ell$ is a bijection from the set of even ordinals
between $\lambda$ and $\lambda + \omega$ onto the set
$\lambda + \omega$. So, once we have generated all of
the even ordinals, we can generate all of $\lambda + \omega$
by applying $\ell$.

This concludes the argument that Property $\Omega$
extends from $\lambda$ to $\lambda+\omega$.

Since we have defined a
J\'{o}nsson-Tarski algebra $J_{\lambda+\omega}$
on $\lambda+\omega$ for any countable limit ordinal $\lambda$,
and done it in a way so that the table on
$\lambda+\omega$ extends the table on $\lambda$,
and Property $\Omega$ is preserved at each extension stage,
we get a J\'{o}nsson-Tarski algebra $J_{\omega_1}$
on $\omega_1$ satisfying Property $\Omega$.
It follows that the subalgebras of $J_{\omega_1}$
are exactly the unions of sets of the form $\lambda+\omega$.
That is, they are exactly the countable limit ordinals.
This is enough to prove that $J_{\omega_1}$ is J\'{o}nsson.
\end{proof}

\bibliographystyle{plain}

\begin{thebibliography}{10}

\bibitem{coleman}
  Coleman, Eoin,
  {\it J\'{o}nsson groups, rings and algebras.}
  Irish Math.\ Soc.\ Bull.\ No.\ {\bf 36} (1996), 34--45

\bibitem{erdos-hajnal1}
  Erd\H{o}s, Paul; Hajnal, Andras,
{\it On a problem of B. J\'{o}nsson.}
Bull.\ Acad.\ Polon.\ Sci.\ Sér.\ Sci.\ Math.\ Astronom.\ Phys.\
{\bf 14} 1966 19--23. 

\bibitem{freese-mckenzie}
  Freese, Ralph; McKenzie, Ralph,
  {\it Residually small varieties with modular congruence lattices.}
  Trans.\ Amer.\ Math.\ Soc.\ {\bf 264} (1981), no.~2, 419--430.

\bibitem{hanf}
  Hanf, William P.,
  {\it Representations of lattices by subalgebras (preliminary report),.}
  Bull.\ Amer.\ Math.\ Soc.\ {\bf 62} (1956), 402.
  
 \bibitem{jonsson-tarski}
J\'{o}nsson, Bjarni; Tarski, Alfred,
{\it On two properties of free algebras.} 
Math.\ Scand., {\bf 9}, (1961),  95--101.

\bibitem{kkv1}
  Kearnes, Keith A.; Kiss, Emil W.; Valeriote, Matthew A.,
  {\it Minimal sets and varieties.}
  Trans.\ Amer.\ Math.\ Soc.\ 350 (1998), no. 1, 1--41.

\bibitem{kearnes-szendrei}
  Kearnes, Keith A.; Szendrei, \'{A}gnes,
  {\it A characterization of minimal locally finite varieties.}
  Trans.\ Amer.\ Math.\ Soc.\ {\bf 349} (1997), no. 5, 1749--1768.

\bibitem{kearnes-szendrei2}
  Kearnes, Keith A.; Szendrei, \'{A}gnes,
  {\it The residual character of strictly simple term minimal algebras.}
  Algebra Universalis {\bf 42} (1999), 269--292.
  
\bibitem{keisler-rowbottom}
  Keisler, H. Jerome; Rowbottom, Frederick,
  {\it Constructible sets and weakly compact cardinals.}
  Preliminary report.
  Notices of the Amer.\ Math.\ Soc.\ {\bf 12} (1965), 373--374.
  
\bibitem{mckenzie}
  McKenzie, Ralph,
  {\it On semigroups whose proper subsemigroups have lesser power.}
  Algebra Universalis 1 (1971), no. 1, 21--25.

\bibitem{olshanskii}
Ol'shanskii, Alexander Yu.,
{\it On a geometric method in the combinatorial group theory.}
Proceedings of the International Congress of Mathematicians,
Vol. 1, 2 (Warsaw, 1983), 415--424, PWN, Warsaw, 1984.

\bibitem{scott}
  Scott, William R.,
  {\it Groups and cardinal numbers.}
  Amer.\ J.\ Math.\ {\bf 74}, (1952). 187--197. 



\bibitem{shelah1}
  Shelah, Saharon,
  {\it On a problem of Kurosh, J\'{o}nsson groups, and applications.}
  Word problems, II (Conf. on Decision Problems in Algebra, Oxford, 1976),
  pp. 373--394, Stud. Logic Foundations Math., 95,
  North-Holland, Amsterdam-New York, 1980.
  

\bibitem{szendrei1}
  Szendrei, \'{A}gnes, 
{\it Strongly abelian minimal varieties.}
Acta Sci.\ Math.\ (Szeged) {\bf 59} (1994), no.~1-2, 25--42.

\bibitem{szendrei-tm}
  Szendrei, \'{A}gnes,
  {\it Term minimal algebras.}
  Algebra Universalis 32 (1994), no.~4, 439--477.
  
\bibitem{szendrei2}
Szendrei, \'{A}gnes,
{\it Maximal non-affine reducts of simple affine algebras.}
Algebra Universalis {\bf 34} (1995), no.~1, 144--174.

\bibitem{szendrei3}
  Szendrei, \'{A}gnes,
  {\it Idempotent algebras with restrictions on subalgebras.}
  Acta Sci.\ Math.\ (Szeged) {\bf 51} (1987), no.~1-2, 251--268.

\bibitem{szendrei4}
  Szendrei, \'{A}gnes,
  {\it Every idempotent plain algebra generates a minimal variety.}
  Algebra Universalis {\bf 25} (1988), no.~1, 36--39.
  
\bibitem{whaley}
Whaley, Thomas P.,
{\it Large sublattices of a lattice.}
Pacific J.\ Math.\ {\bf 28} 1969, 477--484. 

\end{thebibliography}

\end{document}